\documentclass[letterpaper, 10 pt, conference]{IEEEtran}  

\usepackage{mathptmx}       
\usepackage{makeidx}         
\usepackage{graphicx}        
\usepackage{multicol}        
\usepackage{amsmath,amssymb}
\usepackage{amsthm}
\usepackage{euscript}

\usepackage{hyperref}
\hypersetup{
    colorlinks=true,
    linkcolor=blue,
    filecolor=blue,      
    urlcolor=blue,
    citecolor=blue,
}

\usepackage{url}
\usepackage{mathrsfs}
\usepackage{array}\pdfminorversion=4
\usepackage{gatechamberlab}
\usepackage{epstopdf}
\usepackage{multirow}
\usepackage{subfig}
\usepackage{siunitx}
\IEEEoverridecommandlockouts                              

\title{\LARGE \bf
Phase Uncertainty to State Stability of Continuous Periodic Orbits
}

\author{Shishir Kolathaya
\thanks{Shishir Kolathaya is with the School of Mechanical and Civil Engineering, California Institute of Technology,
        Pasadena, CA, USA
        {\tt\small sny@caltech.edu}}%
}

\begin{document}
\maketitle
\thispagestyle{empty}
\pagestyle{empty}

\begin{abstract}
The paper shows sufficiency conditions for stability of continuous periodic orbits under phase uncertainty. Phase based uncertainty is a trait of bipedal walking robots, where the desired trajectories are parameterized by a monotonous function. This monotonous function, called the phase variable, is often affected by intermittent perturbations due to noisy sensors. We will mainly focus on continuous periodic orbits obtained via parameterized trajectories, and then analyze their stability properties under a noisy phase estimation. In other words, our focus is on examples where phase variables are difficult to compute, and therefore are imperfect. We will show that stable periodic orbits subject to phase based uncertainty are input to state stable.
\end{abstract}

\section{Introduction} This article provides a proof for phase-uncertainty-to-state stability of continuous periodic orbits. Phase variables appear extensively in the field of locomotion pattern generators for bipedal robots (see \figref{fig:taucalf}). See \cite{kolathaya2016time,hscc17running} for a brief overview on phase uncertainty. The primary purpose of phase variables is to modulate the desired trajectories of the actuated joints of the robot from start to finish. Choosing a state dependent phase variable renders the trajectory tracking control law autonomous. This state dependency also results in the injection of perturbations into the system via noisy sensory feedback. Therefore stability properties of walking behaviors under an imperfect phase determination is of interest to us.

In this manuscript we will establish preliminary results on stability of continuous periodic orbits under imperfect phase determinations (phase uncertainty). See \cite{TAC:kolathayaPSS} for a detailed analysis on stability of hybrid periodic orbits under phase uncertainties. Also see \cite{kolathaya2016parameter,kolathaya2015parameter} for overview on parameter-uncertainty-to-state stability of hybrid periodic orbits.

The paper is structured as follows. Section II will introduce a brief overview on feedback control laws used to realize stable periodic orbits. Section III will introduce the notion of {\it input to state stability} (ISS), and then the notion of {\it phase uncertainty to state stability}. Finally, Section IV will introduce the main theorem of the paper demonstrating that periodic orbits under phase based uncertainties can be rendered phase to state stable.

\begin{figure}[ht!]
 \centering
 \includegraphics[height=5cm]{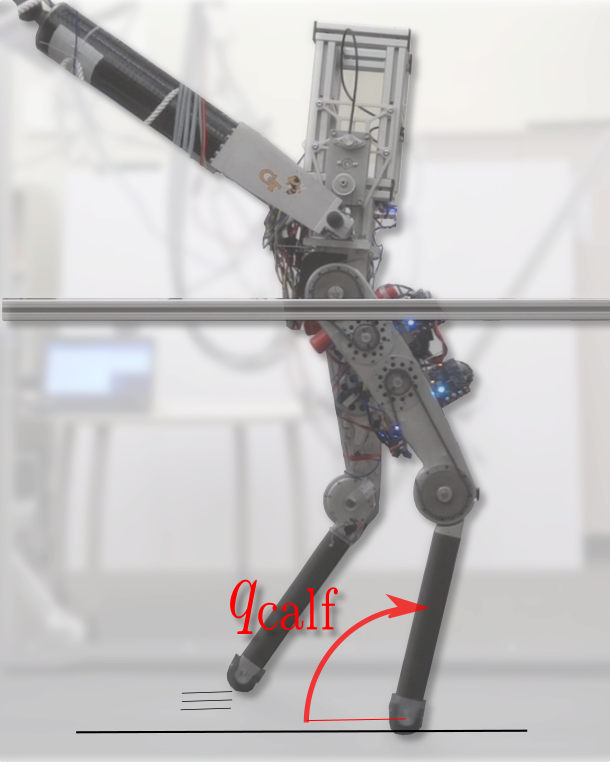}
 \caption{The red arrow indicates a specific phase variable candidate; in this case the calf angle. As the robot progresses throughout a step, the calf angle monotonically increases. This is scaled appropriately with an offset to modulate from $0$ to $1$ for each step.}
 \label{fig:taucalf}
\end{figure}

  \section{Feedback Control}
\label{sec:control}

The goal of this section is to define the set of outputs given the state $x$. Virtual constraints \cite{WGCCM07} consist of a vector of actual outputs given as $y^a:T\ConfigSpace \to \mathbb{R}^k$, and a vector of desired outputs given as $y^d :\R_{\geq 0} \to \mathbb{R}^k$. Here $y^d$ is a map from the positive reals, and can thus be parameterized by a phase (or time) variable $\tau:T\ConfigSpace \to \R_{\geq 0}$ (or $\tau:\R_{\geq 0} \to \R_{\geq 0}$ for time based). By adapting a feedback linearizing controller, we can drive the relative degree one outputs (velocity outputs)
\begin{align}\label{eq:d1o}
y_{1} (q,\dq) = y^a_{1}(q,\dq) - y^d_{1}(\tau,\alpha) \in \R^{k_{1}},
\end{align}
and relative degree two outputs (pose outputs)
\begin{align}
\label{eq:d2o}
y_{2} (q) = y^a_{2}(q) - y^d_{2}(\tau,\alpha) \in \R^{k_{2}},
\end{align}
to zero, with $\alpha$ denoting the parameters of the desired trajectory. ${k_{1}}+{k_{2}}=k$. These outputs are generally called \emph{virtual constraints} \cite{WGCCM07}. Normally, the phase variable, $\tau$, for relative degree two outputs is a function of the configuration $\tau(q)$. Walking gaits, viewed as a set of desired periodic trajectories, are modulated as functions of a phase variable to eliminate the dependence on time \cite{villarreal2014survey}. In this case, the velocity (relative degree one) outputs are: $y_{1} (q,\dq) = y^a_{1}(q,\dq) - y^d_{1}(\alpha)$, where the parameterization w.r.t. $\tau$ is absent. In terms of the states, $x$, we can define the outputs as follows \cite{ames2014human}:
\begin{align}
y_{1} (x) &= y^a_{1}(x) - y^d_{1}(\alpha) \nonumber \\
y_{2} (x) &= y^a_{2}(x) - y^d_{2}(\tau,\alpha) .
\end{align}

\newsec{Feedback Linearization.}\label{subsec:fblin}
The feedback linearizing controller that drives the purely state dependent outputs $y_{1}\to 0$, $y_{2} \to 0$ is given by:
\begin{align}
\label{eq:fblin}
 u = \begin{bmatrix}L_{g} y_{1} \\  L_{g} L_{f} y_{2}\end{bmatrix} ^{-1} \left ( - \begin{bmatrix} L_{f} y_{1} \\ L_{f}^2 y_{2} \end{bmatrix} + \mu \right ),
\end{align}
where $L_{f},L_{g}$ denote the Lie derivatives and $\mu$ denotes the auxiliary input applied after the feedback linearization.

Note that any effective tracking controller will theoretically suffice (the experimental implementation uses PD control \cite{Hereid_etal_2016}). We therefore employ a control Lyapunov function (CLF) based controller that can drive the following outputs to zero
\begin{align}\label{eq:notationeta}
 \eta = \begin{bmatrix} y_{1} \\ y_{2} \\ \dot y_{2} \end{bmatrix}.
\end{align}
If the system has outputs with more relative degrees of freedom, then $\eta$ can be accordingly modified. Applying the controller \eqref{eq:fblin} results in the following output dynamics:
\begin{align}\label{eq:FGandall}
 \dot{\eta} =  \underbrace{\begin{bmatrix}
				      0 & 0 & 0 \\
				      0 & 0 & I_{k_{2} \times k_{2}} \\
				      0 & 0 & 0
				      \end{bmatrix}}_{F} \eta + \underbrace{\begin{bmatrix}  I_{k_{1} \times k_{1}} & 0 \\ 0 & 0 \\ 0 & I_{k_{2} \times k_{2}} \end{bmatrix}}_{G} \mu, 
\end{align}
where $k_{1}+k_{2} =k$. $k_{1}$ is the size of the velocity outputs $y_{1}$, and $k_{2}$ is the size of the relative degree two outputs $y_{2}$. The dimension of the outputs $k$ is typically equal to the number of control inputs $m$. 

The auxiliary control input $\mu$ is chosen via control Lyapunov functions (CLFs) that drives $\eta \to 0$. More on CLFs is explained toward the end of this section.

\newsec{Zero Dynamics.} When the control objective is met such that $\eta = 0$ for all time then the system is said to be operating on the {\it  zero dynamics surface}:
\begin{align}
\label{eq:zerodyn}
 \ZD{} = \{ (q,\dq) \in \pi_x(\Domain_{}) | \eta = 0 \},
\end{align}
for the domain $\Domain$. Further, by relaxing the zeroing of the velocity output $y_{1}:T_q\ConfigSpace \to \R^{k_1}$, we can realize {\it partial zero dynamics} \cite{ames2014human}:
\begin{align}
\label{eq:zerodyn2}
 \PZD{} = \{ (q,\dq) \in \pi_x(\Domain) | y_{2} = 0 , L_f y_{2} = 0 \}.
\end{align}
The humanoid DURUS, has feet and employs ankle actuation to propel the hip forward during the continuous dynamics. The relaxation assumption is implemented on the hip velocity, resulting in {\it partial zero dynamics}.
For the running robot DURUS-2D, since the feet are underactuated, purely relative degree two outputs are picked that result in {\it full zero dynamics } of the system: $\PZD{}=\ZD{}$,
due to the absence of velocity outputs.

The zero dynamics are characterized by the zero dynamic coordinates $z \in \R^{2n-k_1-2k_2}$, which when combined with the normal coordinates $\eta$ form the transformed state space for the full order dynamics: 
\begin{align}
\label{eq:dynamicseta}
 \dot{\eta} &= F \eta + G \mu \nonumber \\
 \dot{z}  &= \Psi (\eta ,z ), 
\end{align}
 When the transverse and the zero dynamics are combined together, we get the full order dynamics. Based on this construction, we have the diffeomorphism $\Phi : \R^{2n} \to \R^{2n}$ that maps from $x = (\q,\dq)$ to $(\eta,z)$. The diffeomorphism can be divided into parts:
 \begin{align}
  \label{eq:diffeomorphism}
  \Phi (x) = \left [  \begin{array}{c} 
   \Phi_{1}(x) \\ \hline \Phi_{2}(x) \\ \hline \Phi_{3}(x)
  \end{array} \right ] = \left [ \begin{array}{c}
			y_{1} (\q,\dq) \\\hline y_{2} (\q) \\ \dot y_{2} (\q,\dq) \\\hline z (\q,\dq)
		    \end{array} \right ].
 \end{align}
 Similarly, the outputs can also be divided into two parts: 
 \begin{align}\label{eq:etasplit}\eta = \begin{bmatrix} y_{1} \\ \eta_{2}\end{bmatrix}, \: \rm{where} \quad  \eta_{2}=\begin{bmatrix} y_{2}\\ \dot y_{2}\end{bmatrix}.\end{align}

\newsec{Control Lyapunov Function.} We will define the {\it control Lyapunov function }(CLF), and the {\it rapidly exponentially stabilizing control Lyapunov function} (RES-CLF) as follows:
\begin{definition}
For the system \eqref{eq:dynamicseta}, a continuously differentiable function $V: \R^{k_1+2k_2} \to \R_{\geq 0}$ is an \textbf{ exponentially stabilizing control Lyapunov function  (ES-CLF)} if there exist positive constants $\underbar c,\bar c,c > 0$ such that for all $\eta,z$.
\begin{align}
\label{eq:es-clf}
 & \underbar c \|\eta\|^2 \leq V(\eta) \leq \bar c \|\eta\|^2 \nonumber \\
 & \inf_{u \in \mathrm{U} }  [ L_f V(\eta,z) + L_g V(\eta,z) u + c V(\eta) ] \leq 0.
\end{align}
\end{definition}
Here $L_f,L_g$ are the Lie derivatives. We can accordingly define a set of controllers which render exponential convergence of the transverse dynamics:
\begin{align}
\K(\eta,z) = \{ u \in \ControlInput: L_f V(\eta,z) + L_g V(\eta,z) u + c V(\eta)  \leq 0 \}, 
\end{align}
which has the control values that result in $\dot{V} \leq - c V$.

\newsec{RES-CLF.} 
We can impose stronger bounds on convergence
 by constructing a {\it rapidly exponentially stabilizing control lyapunov function (RES-CLF)} that can be used to stabilize the output dynamics at a rapid rate through a user defined $\epsilon > 0$. 
 \begin{definition}
  For the family of continuously differentiable functions, $V_\epsilon : \R^{k_1+2 k_2} \to \R_{\geq 0} $ is a \textbf{rapidly exponentially stabilizing control Lyapunov function (RES-CLF)}, if there exist positive constants $c_1,c_2,c_3 > 0$ such that  for all $0 < \epsilon < 1$ and for all $\eta,z$,
\begin{align}
\label{eq:boundsonVe}
& c_1 \|\eta\|^2 \leq V_\epsilon (\eta) \leq \frac{c_2}{\epsilon^2} \| \eta \|^2 ,  \\
 & \inf_{u \in \ControlInput} \lbrack L_f V_\epsilon(\eta,z) + L_g V_\epsilon (\eta,z) u  + \frac{c_3}{\epsilon} V_\epsilon (\eta) \rbrack \leq 0. \nonumber
\end{align}
 \end{definition} 
Therefore, we can define a class of controllers $K_\epsilon$:
\begin{align}
\label{eq:classinputg}
\K_\epsilon(\eta,z) = \{ u \in \ControlInput : L_f V_\epsilon (\eta,z) + L_g V_\epsilon (\eta,z)  u +\frac{c_3}{\epsilon} V_\epsilon (\eta) \leq 0 \},
\end{align}
which yields the set of control values that satisfies the desired convergence rate. Note that the set depends on the domain, due to the dependency on $f ,g $. Therefore, the set with the domain subscript is denoted as $\K_{\epsilon}$.



\newsec{Time Dependent  RES-CLF.}  
Given \eqref{eq:d2o}, if the desired outputs are parameterized by the time based phase variable instead of the state based phase variable, we have the following output representation
\begin{align}
\label{eq:d1ot}
y^t_{1} (q,\dq) = y^a_{1}(q,\dq) - y^d_{1}(\tau(t),\alpha),
\end{align}
for velocity outputs and
\begin{align}
\label{eq:d2ot}
y^t_{2} (q) = y^a_{2}(q) - y^d_{2}(\tau(t),\alpha),
\end{align}
for relative degree two (pose) outputs. The outputs are derived from \eqref{eq:d1o},\eqref{eq:d2o} where the phase is now dependent on time $\tau(t)$. With the absence of parameterization for the velocity outputs, we have : $y^t_{1}  (q,\dq) = y_{1} (q,\dq)$. The resulting output dynamics is obtained by taking the derivative:
\begin{align}
\label{eq:timeoutputs}
 \dot{y}^t_{1} (q,\dq) &=  L_f y^a_{1}(q,\dq) + L_g y^a_{1}(q,\dq) u  - \dot{y}^d_{1}(\tau(t),\dot \tau(t),\alpha) \\
\ddot{y}^t_{2} (q,\dq) &=  L_f^2 y^a_{2}(q) + L_g L_f y^a_{2}(q)u - \ddot{y}^d_{2}(\tau(t),\dot \tau(t), \ddot \tau(t), \alpha). \nonumber 
\end{align}

In order to drive the time dependent outputs $$\eta_t (x) = \begin{bmatrix} y^t_{1} (x) \\ \eta_{2,t} (x) \end{bmatrix} = \begin{bmatrix} y^t_1 (q,\dq) \\  y^t_2 (\q) \\ \dot{y}^t_2 (q,\dq) \end{bmatrix} \to 0,$$ we can choose $u$  via \textbf{ time dependent RES-CLF}s (similar to \eqref{eq:classinputg}) in the following manner:
\begin{align}
\label{eq:classinputgt}
 	\K^t_\epsilon(\eta_t,z_t)& = \{u_t \in \ControlInput :  L_f V^t_\epsilon(\eta_t) + L_g V^t_\epsilon(\eta_t) u_t + \frac{\gamma}{\epsilon} V^t_\epsilon(\eta_t) \leq 0 \}.
\end{align}
This set of controllers with the domain representation will hence be denoted as $\K^t_{\epsilon}$. A particular control solution that belongs to the set $\K^t_\epsilon$ can be obtained via feedback linearization:
\begin{align}
\label{eq:fblint}
 u_t = \begin{bmatrix} L_g y^a_{1} \\  L_g L_f y^a_{2}\end{bmatrix} ^{-1} \left ( - \begin{bmatrix} L_f y^a_1 \\ L_f^2 y^a_2 \end{bmatrix} + \begin{bmatrix}  \dot{y}^d_{1} \\ \ddot{y}^d_{2} \end{bmatrix} + \mu_t \right ),
\end{align}
where $\mu_t$ is the auxiliary time based control input after feedback linearization that can be appropriately chosen. If the desired velocity outputs have no parameterization, then $\dot y^d_1 =0$ (derivative of a constant).
\noindent The time based output dynamics can be written in normal form as
\begin{align}\label{eq:dynamicsetat}
\dot{\eta}_t = F \eta_t + G \mu_t, \:\:\:\: \dot{z}_t = \Psi_t(\eta_t,z_t).
\end{align}
$z_t$ are the set of zero dynamic coordinates normal to $\eta_t$ and has the invariant dynamics $\dot{z}_t = \Psi_t(0,z_t)$. Also note that the matrices $F,G$ are the same as the matrices used for the state based coordinates in \eqref{eq:dynamicseta}. For the time based states, $\eta_t,z_t$, we have the diffeomorphism: $\Phi_t(x)=(\eta_t(x),z_t(x))$.

\newsec{RES-CLFs Obtained From Feedback Linearization.} Due to the difficulty in obtaining Lyapunov functions (in particular, control Lyapunov functions) for nonlinear systems, we use the linear dynamics \eqref{eq:dynamicseta},\eqref{eq:dynamicsetat} obtained from feedback linearization to realize both state and time dependent RES-CLFs. Rapid exponential convergence is obtained by appropriately choosing $\mu,\mu_t$ for \eqref{eq:dynamicseta},\eqref{eq:dynamicsetat} respectively. This is mainly discussed in \cite{TAC:amesCLF} and will be explained in brief here.

Let $F,G$ be defined as in \eqref{eq:dynamicseta},\eqref{eq:dynamicsetat}, and let $P$ be the solution to the CARE (control algebraic Riccati equation)
 \begin{align}
  F^TP + P F -  P G G^T P +  Q = 0,
 \end{align}
 for some $Q=Q^T >0$. Since $\gamma P \leq Q$, where $\gamma = \frac{\lambda_{min}(Q)}{\lambda_{max}(P)} >0$. $\lambda_{min}(.),\lambda_{max}(.)$ denote the minimum and maximum eigenvalues of a given symmetric matrix respectively. By choosing $\epsilon>0$ and letting $P_\epsilon := \begin{bmatrix}
 \frac{1}{\epsilon} I & 0 \\
   0 & \frac{1}{\epsilon} I  
\end{bmatrix} P \begin{bmatrix}
 \frac{1}{\epsilon} I & 0 \\
   0 & \frac{1}{\epsilon} I  
\end{bmatrix}$ and $Q_\epsilon : = \begin{bmatrix}
 \frac{1}{\epsilon} I & 0 \\
   0 & \frac{1}{\epsilon} I  
\end{bmatrix} Q \begin{bmatrix}
 \frac{1}{\epsilon} I & 0 \\
   0 & \frac{1}{\epsilon} I  
\end{bmatrix}$, the following is satisfied
 \begin{align}\label{eq:epcare}
  F^TP_\epsilon + P_\epsilon F - \frac{1}{\epsilon} P_\epsilon G G^T P_\epsilon + \frac{1}{\epsilon} Q_\epsilon =0.
 \end{align}
Motivated by \eqref{eq:epcare}, we can construct the following Lyapunov function:
\begin{eqnarray}
\label{eq:RESCLF}
	 V_{\epsilon}(\eta) =  \eta^TP_{\epsilon}\eta  .
\end{eqnarray}
Differentiating \eqref{eq:RESCLF} yields
\begin{align}
 \label{eq:LfLg}
&\dot{V}_{\epsilon}(\eta) = L_F V_{\epsilon}(\eta) + L_G V_{\epsilon}(\eta) \mu,\nonumber \\
&L_F V_{\epsilon}(\eta)  = \eta^T (F^T P_{\epsilon} + P_{\epsilon}F) \eta, \:\:\: L_G V_{\epsilon}(\eta) = 2\eta^T P_{\epsilon} G.
\end{align}
Here $\mu$ can be picked from the set given below:
\begin{align}
\label{eq:classinputgl}
 V_\epsilon(\eta)& = \eta^T P_\epsilon \eta,  \\
 	\K_\epsilon(\eta)& = \{\mu\in\R^m :  L_F V_\epsilon(\eta) + L_G V_\epsilon(\eta) \mu + \frac{\gamma}{\epsilon} V_\epsilon(\eta) \leq 0 \}.\nonumber
\end{align}
The above controller drives the state dependent outputs rapidly exponentially to zero, $\dot{V}_\epsilon \leq - \frac{\gamma}{\epsilon} V_\epsilon$, which is required by the conditions of \eqref{eq:boundsonVe}. Therefore, $V_\epsilon$ is a valid RES-CLF.
In a similar fashion, $\mu_t$ can be picked from the set given below:
\begin{align}
\label{eq:classinputgtl}
 V^t_\epsilon(\eta_t)& = \eta_t^T P_\epsilon \eta_t,  \\
 	\K^t_\epsilon(\eta_t)& = \{\mu_t\in\R^m :  L_F V^t_\epsilon(\eta_t) + L_G V^t_\epsilon(\eta_t) \mu_t + \frac{\gamma}{\epsilon} V^t_\epsilon(\eta_t) \leq 0 \}.\nonumber
\end{align}

The above controller drives the time dependent outputs rapidly exponentially to zero, $\dot{V}^t_\epsilon \leq - \frac{\gamma}{\epsilon} V^t_\epsilon$.
$L_G, L_F$ are the Lie derivatives that are similar to \eqref{eq:LfLg} (see \cite{TAC:amesCLF}):
\begin{align}
L_F V^t_\epsilon &= \eta_t^T ( F^T P_\epsilon + P_\epsilon F) \eta_t  , \:\:\: L_G V^t_\epsilon = 2 \eta_t^T P_\epsilon G.
\end{align}

For limiting the use of notations, the classes of controllers defined in \eqref{eq:classinputgt} and \eqref{eq:classinputgtl} are both denoted by $\K^t_\epsilon$ with the difference being the dependency on the number of arguments. Therefore, a standard time based RES-CLF without substituting \eqref{eq:fblint} would be denoted by $\K^t_\epsilon(\eta_t,z_t)$, and the time dependent RES-CLF that utilizes the auxiliary input $\mu_t$ after substituting \eqref{eq:fblint} would be denoted by $\K^t_\epsilon(\eta_t)$. Similarly, a state based RES-CLF without the substitution of \eqref{eq:fblin} is denoted by $\K_\epsilon(\eta,z)$, and with the substitution would be denoted by $\K_\epsilon(\eta)$. To summarize, the main control input $u$ is a function of two arguments, while the auxiliary control input $\mu$ is a function of only one, the normal coordinates $\eta$.

  \section{Preliminaries on Input to State Stability}
\label{app:iss}
In this appendix we will introduce basic definitions and results related to input-to-state stability (ISS) for a general nonlinear system. See \cite{sontag2008input} for a detailed survey on ISS. Consider the following differential equation:
\begin{align}
\label{eq:system}
\dot{x} =& f(x,d),
\end{align}
with $x$ taking values in Euclidean space $\R^n$, the input $d \in \R^m$ for some positive integers $n,m$. The mapping $f:\R^n \times \R^m \to \R^n $ is considered Lipschitz continuous and $f(0,0)=0$. Note that the dimension of the state for the robots considered are of dimension $2n$, $x\in\R^{2n}$. 

It can be observed that the input considered here is $d$. Therefore, the construction is such that a stabilizing controller $u=k(x)$ is applied. Any deviation from this stabilizing controller can be viewed as $k(x)+d$, with $d$ being the new disturbance input. 
We assume that $d:\R_{\geq 0} \to \R^m$ is a Lebesgue measurable function of time: $\|d\|_\infty = \mathrm{ess.} \: \mathrm{sup}_{t \geq 0} \| d(t) \| < \infty$. We can denote this space of Lebesgue measurable functions as $\mathbb{L}^m_\infty$, and therefore $d\in\mathbb{L}^m_\infty$.

\newsec{Class $\mathcal{K}_\infty$ and $\mathcal{KL}$ functions.} 
A class $\mathcal{K}_\infty$ function is a function $\alpha:\R_{\geq 0} \to \R_{\geq 0}$ which is continuous, strictly increasing, unbounded, and satisfies $\alpha(0) = 0$, and a class $\mathcal{KL}$ function is a function $\beta:\R_{\geq 0} \times \R_{\geq 0} \to \R_{\geq 0}$ such that $\beta(r, t) \in \mathcal{K}_\infty$ for each $t$ and $\beta(r, t) \to 0$ as $t \to \infty$.

%

We can now define input to state stability for system \eqref{eq:system}.
\gap
\begin{definition}
 The system \eqref{eq:system} is input to state stable (ISS) if there exists $\beta \in \mathcal{KL}$, $\iota \in \mathcal{K}_\infty$ such that
 \begin{align}\label{eq:ISSmaindefinition}
  |x(t,x_0)| \leq \beta(|x_0|,t) + \iota(\|d\|_\infty), & \hspace{10mm} \forall x_0,d, \forall t \geq 0,
 \end{align}
 and \eqref{eq:system} is considered locally ISS, if the inequality \eqref{eq:ISSmaindefinition} is valid for an open ball of radius $r$, $x_0 \in \mathbb{B}_r (0)$.
\end{definition}
\gap
\begin{definition}
 The system \eqref{eq:system} is exponential input to state stable (e-ISS) if there exists $\beta \in \mathcal{KL}$, $\iota \in \mathcal{K}_\infty$ and a positive constant $\lambda > 0$ such that
 \begin{align}\label{eq:ISSmainexpodefinition}
 |x(t,x_0)| \leq \beta(|x_0|,t) e^{-\lambda t} + \iota(\|d\|_\infty), & \hspace{5.5mm} \forall x_0,d, \forall t \geq 0,
 \end{align}
 and \eqref{eq:system} is considered locally e-ISS, if the inequality \eqref{eq:ISSmainexpodefinition} is valid for an open ball of radius $r$, $x_0 \in B_r (0)$.
\end{definition}
\gap


 \begin{definition}
  The system is said to hold the asymptotic gain (AG) property if there exists $\iota \in \mathcal{K}_\infty$ such that
  \begin{align}\label{eq:AG}
   \overline \lim_{t \to \infty} |x(t,x_0)| \leq \iota (\|d\|_\infty), & \hspace{10mm} \forall x_0,d.
  \end{align}
  \end{definition}
\gap  
 \begin{definition}
 The system is said to be zero stable if there exists $\beta \in \mathcal{KL}$ such that:
 \begin{align}\label{eq:ZS}
 |x(t,x_0)| \leq \beta(|x_0|,t), & \hspace{10mm} \forall x_0,d, \forall t \geq 0.
 \end{align}
 \end{definition}
 The AG and ZS property are both pictorially shown in \figref{fig:zsag}.

\newsec{ISS-Lyapunov functions.} We can develop Lyapunov functions that satisfy the ISS conditions and achieve the stability property.
\gap
\begin{definition}
A smooth function $V:\R^n \to \R_{\geq 0}$ is an ISS-Lyapunov function for \eqref{eq:system}  if there exist functions $ \underline \alpha$, $\bar \alpha$, $\alpha$, $\iota \in \mathcal{K}_\infty$ such that
\begin{align}
\label{eq:ISSd}
&\underline \alpha(|x|) \leq V (x)  \leq \bar \alpha (|x|)  \nonumber \\
&\dot{V}(x,d)  \leq - \alpha(|x|) \quad \mathrm{for} |x| \geq \iota (\|d\|_\infty).
\end{align}
\end{definition}
\gap
\noindent The following lemma establishes the relationship between the ISS-Lyapunov function and the ISS of \eqref{eq:system}.
\gap
\begin{lemma}
\label{thm:ISSTheorem}
The system \eqref{eq:system} is ISS if and only if it admits a smooth ISS-Lyapunov function.
\end{lemma}
\gap
\noindent Proof of Lemma \ref{thm:ISSTheorem} is given in \cite{sontag2008input} and in \cite{sontag1989smooth}. In fact the inequality condition can be made stricter by using the exponential estimate:
\begin{align}
\label{eq:ISSdstricter}
\dot{V}(x,d)  \leq - c V(x) + \iota (\|d\|_\infty),	 & \hspace{10mm} \forall x,d.
\end{align}
which is then called the e-ISS Lyapunov function.

\noindent We can also use the AG propery \eqref{eq:AG} to establish ISS:
\begin{lemma}\label{lm:AG}
 The system is ISS if and only if it is zero stable and AG.
\end{lemma}

\begin{figure}
\centering
	\subfloat{
		\includegraphics[width= 0.4\columnwidth]{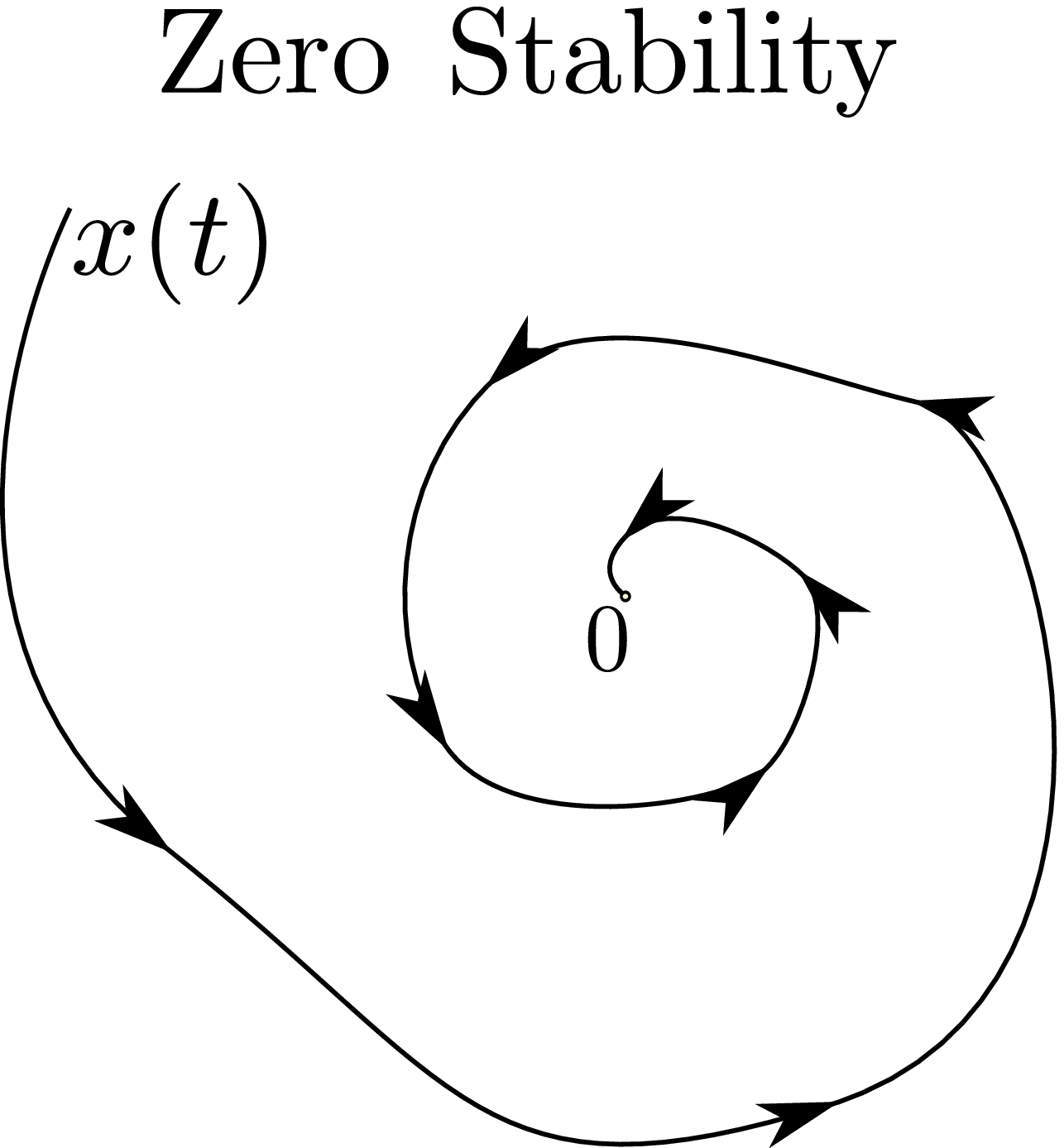}
	}
	\subfloat {
	 \includegraphics[width= 0.4\columnwidth]{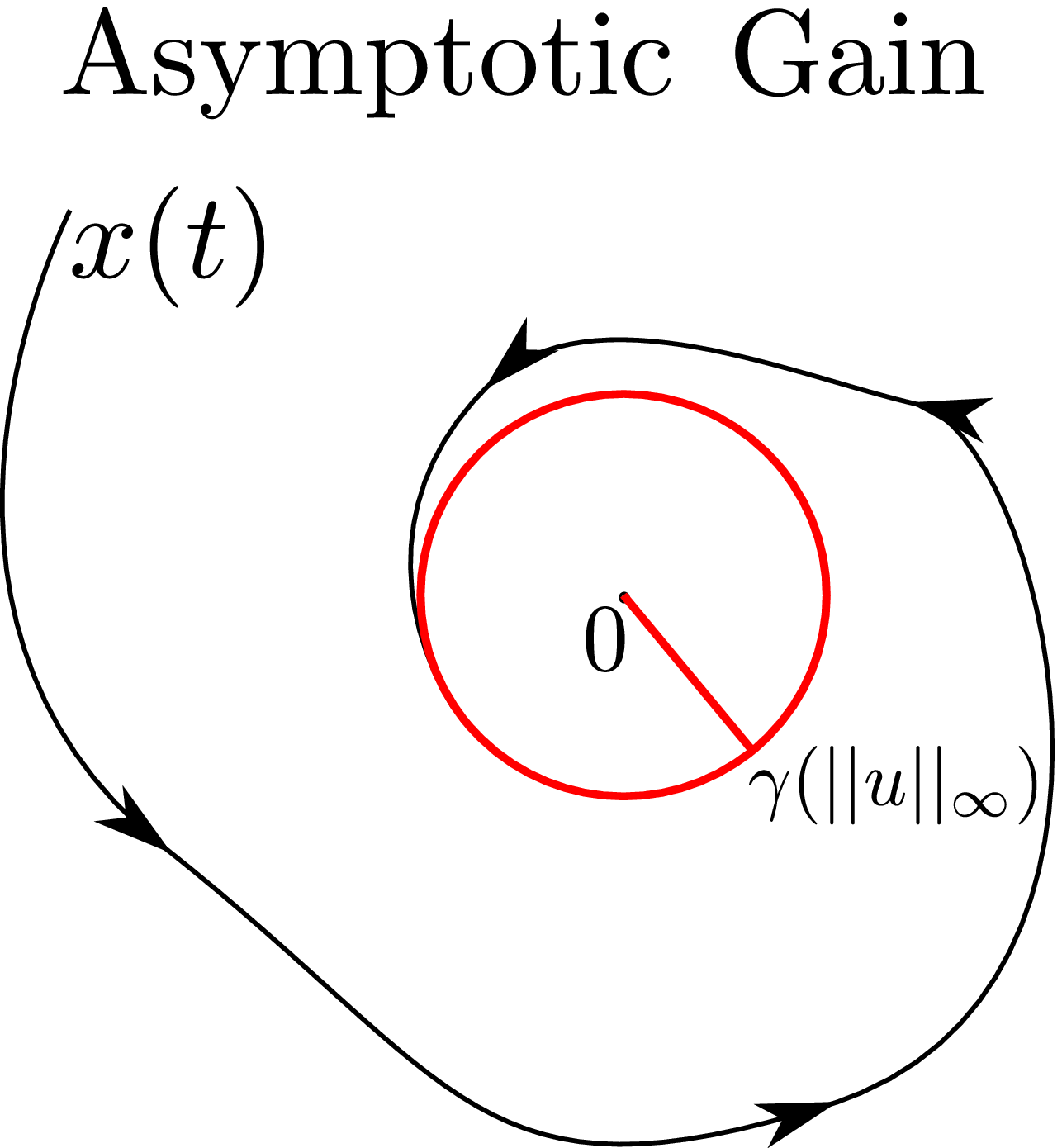}
	}
	\caption{If the system is ISS, zero stability is achieved for a zero input, and asymptotic gain is achieved for a bounded input. }
	\label{fig:zsag}
\end{figure}

  \section{Phase to State Stability}
\label{app:pssc}
We can now define the notion of {\it phase to state stability} for continuous sytems below. A preliminary on ISS is given in Appendix \ref{app:iss}.
Systems considered are of the type 
\begin{align}\label{eq:dynamicsdhere}
 \dot \eta &= F \eta + G \mu(\eta)  + G d \nonumber \\
 \dot z &= \Psi(\eta,z),
\end{align}
where a suitable Lipschitz control law $\mu(\eta) \in \K_{\epsilon}(\eta)$ is applied. Since, the analysis is only for continuous dynamics, the domain notation is ignored. A suitable control law would have been $\mu(\eta)$, which is state based, but instead, a time based control law, $\mu(\eta)+d$, was applied. The time dependency is implicit in the disturbance input $d$.

\gap
\begin{definition}
 Assume a ball of radius $r$ around the origin. The system given by \eqref{eq:dynamicsdhere} is locally  \textbf{phase to $\eta$ stable }, if there exists $\beta \in \classKL$, and $\iota \in \classK_\infty$ such that
  \begin{align}\label{eq:ISSmaindefinitionPhase}
  |\eta(t)| \leq \beta(|\eta(0)|,t) + \iota(\|d\|_\infty), &  \forall \eta(0) \in \mathbb{B}_r(0),\forall d, \forall t \geq 0,
 \end{align}
 and it is locally \textbf{phase to state stable }, if
 \begin{align}
  |(\eta(t),z(t))| \leq \beta(|(\eta(0),z(0))|,t) + \iota(\|d\|_\infty), & \nonumber \\ \forall  \eta(0) \in \mathbb{B}_r(0), \forall d, \forall t \geq 0.  \nonumber
 \end{align}
\end{definition}
\gap

We will first establish phase to $\eta$ stability, and then include the zero dynamics to show phase to state stability.
Based on the asymptotic gain and zero stability property of the system \eqref{eq:dynamicsdhere} w.r.t. the {\it phase uncertainty} $d$, we have the following lemma.
\gap
\begin{lemma}
\label{thm:exponcon}
{\it The class of all Lipschitz continuous feedback control laws $\mu(\eta) \in \K_\epsilon(\eta)$ applied on the system \eqref{eq:dynamicsdhere} yields phase to $\eta$ stability in the continuous dynamics. 
}
\end{lemma}
\begin{proof}
   Proof is provided in \cite{hscc17running} and is straightforward due to the fact that a stabilizing controller $\mu(\eta)$ is sufficient to render the linear system phase to $\eta$ stable. The ultimate bound on $\eta$ can be further obtained as $\frac{4c_2}{\gamma c_1 \epsilon}\|d\|_\infty$. This ultimate upper bound is computed in \cite{hscc17running}.
%
  \end{proof}
\gap


 We can also realize \textbf{exponential phase to state stability} of the continuous dynamics by appending a state based linear feedback law to the time based feedback linearizing control resulting in the dynamics 
 \begin{eqnarray}
 \label{eq:clflinear2prepeat}
 \dot{\eta} &=& F\eta + G \mu (\eta) + G d + G B_y u_s(\eta,z)  \nonumber \\
	\dot{z} &=& \Psi(\eta,z).
\end{eqnarray}


Lemma \ref{thm:exponcon} can now be redefined to obtain exponential phase to state stability.
\gap
\begin{lemma}
\label{thm:exponcon2p}
{\it Given the Lipschitz continuous control laws $\mu(\eta) \in \K_\epsilon(\eta)$, $u_s(\eta,z) \in \K^s_{\epsilon,\bar\epsilon}(\eta,z)$ applied on the system \eqref{eq:clflinear2prepeat} yields exponential phase to $\eta$ stability in the continuous dynamics. 
}
\end{lemma}
  \begin{proof}
Proof is again provided in \cite{hscc17running} which is similar to the proof for Lemma \ref{thm:exponcon}. The ultimate bounds can also be explicitly computed by taking the derivative of the Lyapunov function, $V_\epsilon$ and substituting the upper bounds. The ultimate upper bound on $\eta$ is $\frac{2\bar\epsilon c_2}{c_1^2\epsilon^2}\|d\|_\infty$.
  \end{proof}
\gap


\newsec{Continuous Periodic Orbits.} 
Application of a Lipschitz continuous feedback control law $\mu(\eta) \in \K_\epsilon(\eta)$, $u_s(\eta,z) \in \K^s_{\epsilon,\bar\epsilon}(\eta,z)$ results in the closed loop vector field \eqref{eq:clflinear2prepeat}. Associated with this vector field is a flow that is a function of $\epsilon,\bar\epsilon$ and also the disturbance $d$. Denote the flow as $\phi^{\epsilon,\bar\epsilon,d}_t(\eta,z)$. The flow with zero disturbance ($d(t)\equiv 0$) is periodic with period $T>0$ and a fixed point $(\eta^*,z^*)$ if $\phi^{\epsilon,\bar\epsilon,0}_t(\eta^*,z^*)=(\eta^*,z^*)$. Associated with this periodic flow is the periodic orbit 
\begin{align}
 \mathcal{O} = \{\phi_t(\eta^*,z^*)\in \R^{2n} : 0 \leq t \leq T \} ,
\end{align}
Similarly, we denote the flow of the partial zero dynamics by $\phi^{z}_t(y_1,z)$ and the associated periodic orbit by $\mathcal{O}_\PZ$. The periodic orbit of the partial zero dynamics can embedded into the full order dynamics through the canonical embedding $\Pi_0(y_1,z) = (y_1,0,z)$. Therefore, $\mathcal{O} = \Pi_0(\mathcal{O}_\PZ)$.


For the periodic orbit on the zero dynamics, we have the periodic orbit of the full order dynamics via the canonical embedding $\Pi_0(\mathcal{O}_\PZ) = \mathcal{O}$. 
By defining the norm $\|(\eta,z)\| = \|y_{1}\|+\|\eta_{2}\|+\|z\|$, we can define the distance from the periodic orbit as
\begin{align}
\label{eq:distance}
 \|(\eta,z)\|_\mathcal{O} &= \inf_{(\eta',z')\in\mathcal{O}} \|(\eta,z)-(\eta',z')\|  \\
			  &= \inf_{(y'_{1},z')\in\mathcal{O}_\PZ} \|(z-z')\|+\|(y_{1}-y'_{1})\|  + \|\eta_{2}\| .\nonumber
\end{align}
The continuous dynamics is exponentially stable in each domain if there are constants $r,\delta_1,\delta_2 > 0$ such that if $(\eta_\vi,z_\vi)\in\B_r(\mathcal{O})$, a neighborhood of radius $r$ around $\mathcal{O}$,
it follows that $\|\phi^{\epsilon,\bar\epsilon,d}_(\eta,z)\|_\mathcal{O} \leq \delta_1 e^{-\delta_2 t} \|(\eta_\vi,z_\vi)\|_\mathcal{O}$.
$\|(y_1,z)\|_{\mathcal{O}_\PZ}$, as mentioned in \eqref{eq:distance}, represents the distance between $z$ and nearest point on the periodic orbit $\mathcal{O}_\PZ$. 
Given that the partial zero dynamics has an exponentially stable periodic orbit, there is a Lyapunov function $V_\PZ:\R^{2n} \to \R_{\geq 0}$ such that in a neighborhood $\mathbb{B}_r(\mathcal{O}_\PZ)$ of $\mathcal{O}_\PZ$ (by converse Lyapunov theorem \cite{hauser1994converse}) such that
\begin{align}
\label{eq:zerodp}
c_{4} \|(y_1,z)\|^2_{\mathcal{O}_\PZ} &\leq  V_\PZ(y_1,z) \leq c_{5} \|(y_1,z)\|^2_{\mathcal{O}_\PZ}, \nonumber \\
\frac{\partial V_\PZ}{\partial z} \Psi(y_1,0,z)   +  \frac{\partial V_\PZ}{\partial y_1} \dot{y}_1 &\leq  -c_{6} \|(y_1,z)\|^2_{\mathcal{O}_\PZ}, \nonumber \\
\left | \left | \frac{\partial V_\PZ}{\partial (y_1,z)} \right | \right | & \leq  c_{7} \|(y_1,z)\|_{\mathcal{O}_\PZ}.
\end{align}
Define the composite Lyapunov function: $V_c(\eta,z) = \sigma V_\PZ(y_1,z) +   V_\epsilon (\eta)$, we can establish boundedness of the dynamics of the robot when $\|d\|$ is bounded. In other words, we have the following theorem, which establishes {\it phase to state stability} of periodic orbits in continuous systems.

\gap
\begin{theorem}
\label{thm:fivep}{\it
Given that the periodic orbit $\mathcal{O}_{z}$ of the partial zero dynamics is exponentially stable, and given the controllers $\mu(\eta) \in \K_\epsilon(\eta)$, $u_s(\eta,z) \in \K^s_{\epsilon,\bar\epsilon}(\eta,z)$ applied on \eqref{eq:clflinear2prepeat}, that render the outputs $\eta$ stable w.r.t. $d$, then the periodic orbit $\mathcal{O}=\Pi_0(\mathcal{O}_\PZ)$ obtained from the canonical embedding is {\it exponential phase to state stable }.}
\end{theorem}
\begin{proof} Upper bounds and lower bounds on $V_c$ are given by
\begin{align}
\label{eq:fullvp}
V_c(\eta , z) & \leq
       \max \{\sigma c_5,\frac{c_2}{\epsilon^2} \} ( \|(y_1,z)\|^2_{\mathcal{O}_\PZ} + \|\eta \|^2 ) ,\nonumber \\
V_c(\eta , z) & \geq  \min \{ \sigma c_{4},c_{1} \} ( \|(y_1,z)\|^2_{\mathcal{O}_\PZ} + \|\eta \|^2 ),
\end{align}
Therefore, taking the derivative:

\begin{align}
\label{eq:fullv3p}
\dot{V}_c(\eta , z) &=  \sigma \frac{\partial V_\PZ}{\partial z} \Psi(y_1,0,z)   + \sigma \frac{\partial V_\PZ}{\partial y_1} \dot{y}_1 \dots  \\
		    &\:\:\:\:\:\:\:\:\:+ \sigma \frac{\partial V_\PZ}{\partial z}(\Psi(\eta , z)-\Psi(y_1,0,z))+   \dot{V}_\epsilon (\eta) , \nonumber \\
 & \leq  -\sigma c_6 \|(y_1,z)\|^2_{\mathcal{O}_\PZ} + \sigma c_7 L_q \|(y_1,z)\|_{\mathcal{O}_\PZ} \|\eta_2\| +  \dot{V}_\epsilon (\eta), \nonumber \\
  & \leq  -\sigma c_6 \|(y_1,z)\|^2_{\mathcal{O}_\PZ} + \sigma c_7 L_q \|(y_1,z)\|_{\mathcal{O}_\PZ} \|\eta\| +  \dot{V}_\epsilon (\eta), \nonumber
\end{align}
where $L_q$ is the Lipschitz constant for $\Psi$ in \eqref{eq:dynamicsdhere}.
Substituting for $\dot{V}_\epsilon$ leads to the following expression for the Lyapunov function:
\begin{align}
 \dot{V}_c \leq & - \sigma c_6 \|(y_1,z)\|^2_{\mathcal{O}_\PZ} + \sigma c_7 L_q \|(y_1,z)\|_{\mathcal{O}_\PZ} \|\eta\| \nonumber \\
	        & - \frac{\gamma }{\epsilon} V_\epsilon - \frac{1}{\bar \epsilon} \bar V_\epsilon + 2 \|\eta\| \|P_\epsilon\| \|d\|_\infty \\
	         \dot{V}_c \leq & - \sigma c_6 \|(y_1,z)\|^2_{\mathcal{O}_\PZ} + \sigma c_7 L_q \|(y_1,z)\|_{\mathcal{O}_\PZ} \|\eta\| \nonumber \\
	        & - \frac{\gamma }{\epsilon} V_\epsilon - \frac{1}{\bar \epsilon} c^2_1 \|\eta\|^2 + 2 \|\eta\| \|P_\epsilon\| \|d\|_\infty \nonumber 
\end{align}
With $\bar\epsilon$ small enough, the disturbance can be rejected by the expression $\frac{1}{\bar \epsilon} c^2_1 \|\eta\|^2$ for $\|\eta\|\geq \frac{2\bar\epsilon c_2}{c_1^2\epsilon^2}\|d\|_\infty$;
giving the following result:
\begin{align}
       \dot{V}_c \leq  - \sigma c_6 \|(y_1,z)\|^2_{\mathcal{O}_\PZ} + \sigma c_7 L_q \|(y_1,z)\|_{\mathcal{O}_\PZ} \|\eta\| - \frac{\gamma }{\epsilon} V_\epsilon \nonumber \\
       \hspace{20mm} \quad \rm{for} \quad  \|\eta\| \geq \frac{2\bar\epsilon c_2}{c_1^2\epsilon^2}\|d\|_\infty,
\end{align}
which is the standard inequality for ISS-Lyapunov functions \eqref{eq:ISSd}. Therefore, for exponential convergence, $\sigma$ is picked such that
$c_6 c_1 \frac{\gamma}{\epsilon} - \sigma \frac{c_7^2 L_q^2}{4} > 0$.
%
\end{proof}
\gap
\bibliographystyle{plain}
\bibliography{bibdata.bib}

\end{document}